\documentclass[11pt]{article}
\usepackage{epsfig}
\usepackage{amssymb}
\usepackage{amsthm}
\newtheorem{theorem}{Theorem}
\newtheorem{lemma}{Lemma}
\newtheorem{remark}{Remark}

\def\Frac#1#2{\frac{\displaystyle{#1}}{\displaystyle{#2}}}

\begin{document}

\begin{center}
{\Large
{\rm On bounds for solutions of monotonic first order difference-differential systems}}
\vspace*{0.5cm}

{\large
{\rm Javier Segura}}
\vspace*{0.5cm}

Departamento de Matem\'aticas, Estad\'{\i}stica y Computaci\'on.
Facultad de Ciencias. 
Universidad de Cantabria. 
39005-Santander, SPAIN. 

javier.segura@unican.es

\vspace*{0.5cm}

------------------------------------------------------------------------------------------------

{\bf Abstract}

\end{center}
Many special functions are solutions of first order linear systems $y_n'(x)=a_n(x)y_n(x)+d_n(x)y_{n-1}(x)$, 
$y_{n-1}'(x)=b_n(x)y_{n-1}(x)+e_{n}(x)y_n(x)$. We obtain bounds for the ratios $y_n(x)/y_{n-1}(x)$ and the logarithmic
derivatives of $y_n(x)$ for solutions of monotonic systems satisfying
certain initial conditions. 
For the case $d_n(x)e_n(x)>0$, sequences of upper and lower bounds can 
be obtained by iterating the recurrence relation; for minimal solutions
of the recurrence 
these are convergent sequences. The bounds are related to the Liouville-Green approximation for the associated second order ODEs 
as well as to the asymptotic behavior of the associated three-term recurrence relation as
 $n\rightarrow +\infty$; the bounds are sharp both as a function of $n$ and $x$. Many
special functions are amenable to this analysis, and we give several examples of application:
modified Bessel functions, parabolic cylinder functions,
Legendre functions of imaginary variable and Laguerre functions. New Tur\'an-type inequalities are established
from the function ratio bounds.
Bounds for monotonic
systems with $d_n(x)e_n(x)<0$ are also given, in particular for Hermite and Laguerre polynomials of real positive variable;
in that case the bounds can be used for bounding the monotonic region (and then the extreme zeros). 

\vspace*{0.2cm}


\vspace*{0.2cm}


\begin{center}
-------------------------------------------------------------------------------------------------
\end{center}

\section{Introduction}
\label{intro}

Many special functions, and in particular functions of hypergeometric type, satisfy
first order differential systems of the form
\begin{equation}
\label{DDESta}
\begin{array}{l}
y_{n}^{\prime}(x)=a_n (x) y_{n}(x) + d_n (x) y_{n-1}(x),\\
y_{n-1}^{\prime}(x)=b_n (x) y_{n-1}(x) + e_n (x) y_{n}(x) .
\end{array}
\end{equation}

For the particular case of modified Bessel functions
sharp bounds for function ratios $y_n (x)/y_{n-1}(x)$ and logarithmic derivatives 
$y_n'(x)/y_n(x)$, as well as
Tur\'an-type inequalities were recently obtained in \cite{Segura:2011:BRM}; the key
ingredient in the analysis was the study of the qualitative behavior of the solutions of 
the Riccati equation satisfied by $h_n (x)=y_n (x)/y_{n-1}(x)$, together with the application
of the three-term recurrence relation. 

Ratios of Bessel functions appear in a great number of applications, particularly as
parameters of certain probability distributions (see, for instance, the examples
mentioned in \cite{Segura:2011:BRM}). Other special function ratios are
important in applications. In particular, parabolic cylinder ratios
appear in the study of Ornstein-Uhlenbeck processes (see, for instance \cite{Ali:2005:RFH}),
and other special function ratios (Whittaker, Legendre, Gauss hypergeometric functions) play 
similar roles as well \cite{Bor:2009:HD, Bou:2008:SFS,Ant:2011:EST}. 
In all these applications, a common characteristic is that the functions are real and the variables lie
inside a monotonic region
(region free of zeros).

In this paper, we put the ideas of \cite{Segura:2011:BRM} in a more general context 
and we analyze the qualitative behavior of the Riccati equation associated to the ratio $h_n (x)=y_n(x)/y_{n-1}(x)$,
\begin{equation}
\label{Riccainin}
h_n ^{\prime}(x)=d_n(x)-(b_n(x)-a_n(x))h_n(x)-e_n(x)h_n(x)^2 ,
\end{equation}
in the general case in which the quadratic equation
\begin{equation}
e_n(x)\lambda_n (x)^2 +(b_n(x)-a_n(x))\lambda_n (x) -d_n(x)=0
\end{equation}
has two distinct real roots $\lambda_n^{\pm}(x)$.
This case corresponds to
monotonic systems, with solutions which have one zero at most. As we will see, if the functions $\lambda_n^{\pm} (x)$
are monotonic, they are bounds for the ratios $h_n(x)$ satisfying certain initial value conditions.

We will discuss in detail bounds for systems with $d_n(x)e_n(x)>0$. Upper and lower bounds become available
which are accurate approximations for
large parameters and/or variable and with bounded error (upper and
lower bounds are available). 
The bounds 
given are related to the Liouville-Green approximation
for the associated second order ODE 
as well as to the asymptotic behavior of the associated three-term recurrence relation. 
This explains why the bounds become sharper as the variable $x$ and/or the parameter become large.

We provide several examples (section \ref{posiex}): modified Bessel functions,
parabolic cylinder functions, Associated Legendre functions of imaginary argument  and
Laguerre functions of negative argument.  From these bounds, a good number of new Tur\'an-type inequalities are obtained.
Tur\'an-type properties for special functions have received a considerable attention in recent years; just to cite
five different groups of researchers, we mention  \cite{Alz:2009:ATT,Bar:2008:TTI,Bar:2009:ANT,Laf:2006:OST,Segura:2011:BRM} 
(see also references cited therein).

We also give two examples of applications of the methods for the case $d_n(x)e_n(x)<0$ (section \ref{monocondit}) and use
these results for bounding function ratios for Laguerre and Hermite polynomials in the real axis (but outside the oscillatory region).
These bounds can be used for bounding the oscillatory region and, therefore, for bounding the extreme zeros.

In addition to direct applications in several areas, particularly in statistics and stochastic processes,
the bounds on function ratios have implications in the construction of numerical algorithms. These techniques provide
bounds for the region of computable parameters of a given function within the overflow and underflow
limitations, and they also provide bounds for the condition numbers of the functions (see section
\ref{PCFs} for the case of Parabolic Cylinder Functions). Additionally, as
discussed for the particular case of modified Bessel functions \cite{Segura:2011:BRM}, 
the bounds are useful for accelerating the convergence of certain continued fraction representations
which are used in numerical algorithms; for instance, the algorithms in \cite{Gil:1998:CEP,Segura:1998:PCF} could be improved
by using the bounds of sections \ref{PCFs} and \ref{OLFs} for accelerating the convergence.

\section{Qualitative behavior of Riccati equations}
\label{desiccation}

In this section
we deal with first order differential systems (\ref{DDESta}) 
with differentiable coefficients. 
We consider the ratio $h_n(x)=y_n(x)/y_{n-1}(x)$ satisfying the Riccati equation
\begin{equation}
\label{Riccaini}
h'(x)=d(x)-(b(x)-a(x))h(x)-e(x)h(x)^2 .
\end{equation}
The label $n$, which is common for $h$ and the coefficients $a$, $b$, $d$ and $e$,
has been dropped in (\ref{Riccaini}) for simplicity and because the analysis in this section
is valid for any system, depending or not on a parameter $n$. The explicit dependence on $n$ will be
recovered in the next section. 

We have $h'(x)=0$ when $h(x)=\lambda^{\pm}(x)$ with
\begin{equation}
\label{root}
\begin{array}{l}
\lambda^{\pm} (x) = \mbox{sign}(e(x)) R(x)\left[-\eta(x)\pm \sqrt{\eta (x)^2+s}\right],\\
R(x)=\sqrt{\left|\Frac{d(x)}{e(x)}\right|},\eta(x)=\Frac{b(x)-a(x)}{2\sqrt{\left|d(x)e(x)\right|}},
s=\mbox{sign}(d(x)e(x)) ,
\end{array}
\end{equation}
We consider the case with real roots $\lambda^{\pm} (x)$.
Two distinct situations may occur: either $d(x)e(x)>0$, or $d(x)e(x)<0$ but
$|\eta (x)|>1$.

The
condition $d(x)e(x)>0$ generally holds in the whole maximal interval of continuity of the
functions because the coefficients $d(x)$ and $e(x)$ do not change sign under very general conditions
(see, for instance, \cite[lemma 2.1]{Segura:2002:ZSF}\footnote{All that is required is that the system is 
satisfied by two independent sets of functions}). Contrarily, when $d (x) e (x)<0$ the condition $|\eta (x)|>1$
may hold only for a limited range of the variable $x$.
In the first case ($d (x)e (x)>0$) 
$h(x)$ may have one zero or one singularity, but not both (\cite[lemma 2.4]{Segura:2002:ZSF}), 
while in the second $h(x)$ may have both a zero and a singularity (\cite[Theorem 2.1]{Gil:2003:CZT}). 
We analyze the case $d(x)e(x)>0$ and assuming that no change of sign of $h(x)$ occurs. For the case $d(x)e(x)<0$, 
as the examples in section \ref{monocondit}
will show, similar arguments can be applied.

In the sequel, we consider $d(x)e(x)>0$. Without loss of generality, we take $d(x)>0$, $e(x)>0$ 
and then $\lambda^{+}(x)>0$ and $\lambda^{-}(x)<0$; if $d(x)<0$, $e(x)<0$ we can consider the replacement $y\rightarrow -y$ or
$w\rightarrow -w$.
In the next results, $(a,b)$ is an interval where $h(x)$ and the coefficients of the system are differentiable; $a$ or $b$ could be $+\infty$ or
$-\infty$. Depending on the value of $h(x)$ at $a^+$ or $b^-$ different bounds can be established.
First we consider $h(a^+)>0$. We enunciate three results and give a common proof.

\begin{lemma}
\label{posiposi}
If $h(a^+)>0$ then $h(x)>0$ in $(a,b)$
\end{lemma}

\begin{theorem}
\label{casopos}
If $h(a^+)>0$, $\lambda^{+}(x)$ is monotonic and 
$h^{\prime}(a^+)\lambda^{+\prime}(a^+)>0$ then $(h(x)-\lambda^{+}(x))\lambda^{+ \prime}(x)<0$ in $(a,b)$.
\end{theorem}

\begin{theorem}
\label{casopos2}
If $h(a^+)>0$, $\lambda^{+}(x)$ is monotonic and $h^{\prime}(a^+)\lambda^{+ \prime}(a^+)<0$ then either $h(x)$ reaches one 
relative extremum at $x_e\in (a,b)$ (a minimum if $\lambda^{+ \prime}(x)>0$ and a
maximum if $\lambda^{+ \prime}(x)<0$) or $(h(x)-\lambda_{+}(x))\lambda^{+ \prime}(x)>0$ in $(a,b)$.
\end{theorem}

\begin{proof}
If $h(a^+)>0$, then $h(x)$ can not change sign continuously: it can not 
become zero because $h'(x)>0$ if $0\le h(x)<\lambda^{+}(x)$. On the other hand, 
it can not change sign discontinuously;
for this, starting with $h(a^+)>0$, a value $x_{\infty}\in (a,b)$ should exist such that $h(x_{\infty}^{-})=+
\infty$ but this is not possible because $h'(x)<0$ if $h(x)>\lambda^{+}(x)$. 

Now, we consider that $\lambda^{+}(x)$ is monotonic. We take the case $\lambda^{+ \prime}(x)>0$; the
case $\lambda^{+ \prime}(x)<0$ is analogous.

Assume first that $h^{\prime}(a^+)>0$; using (\ref{Riccaini}) this means that
$0<h (a^+)< \lambda^{+}(a^+)$. And then, necessarily $h (x)< \lambda^{+}(x)$ in $(a,b)$. Indeed, because
$\lambda^{+}(x)$ is monotonically increasing and the graph of $h(x)$ is below the graph of $\lambda^{+}(x)$ close to $x=a$, 
the graph of $h(x)$ may touch the graph of $\lambda^{+}(x)$ at $x=x_e$ only if the first one has a larger slope
at $x_e$, that is, if $h^{\prime}(x_e)>\lambda^{+ \prime}(x_e)>0$; but if 
$h(x_e)=\lambda_{+}(x_e)$ then $h'(x_e)=0$.
  
If, contrarily, $h^{\prime}(a^+)<0$
then the graph of $h(x)$ lies above the graph of $\lambda^+ (x)$ close to $x=a$ and there are two possibilities:
either it remains above $\lambda^{+}(x)$ in all the interval or there is a point $x_e\in (a,b)$ where 
$h(x_e)=\lambda^{+}(x_e)$ and $h'(x_e)=0$. Then the graph of $h(x)$ crosses the graph of $\lambda^{+}(x)$,
which is an increasing function, and $h'(x)>0$ for all $x>x_e$. Therefore there is a minimum at $x_e$.
\end{proof}

Figure 1 illustrates the situations described in Theorems \ref{casopos} and \ref{casopos2}.

\vspace*{0.7cm}
\begin{center}
\epsfxsize=8cm \epsfbox{positive.eps}
\end{center}
{\footnotesize
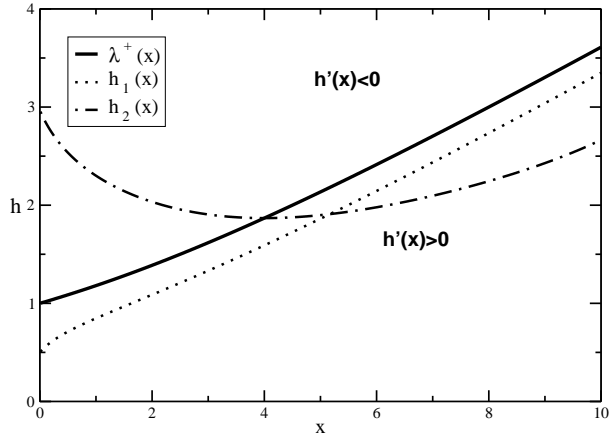
Figure 1: The characteristic root $\lambda^+ (x)$ divides the plane in two regions: $h'(x)>0$ if $0<h(x)<\lambda^+ (x)$
and $h'(x)<0$ if $h(x)>\lambda^+ (x)$. The graph of $h_1 (x)$ corresponds to the situation described in Theorem \ref{casopos} while
$h_2 (x)$ corresponds to Theorem \ref{casopos2} when an extremum is reached.}

If, differently from theorems \ref{casopos} and \ref{casopos2}, we have
 $h(a^+)<0$ then $h(x)$ may change sign once. But if it does not change sign and $h(b^-)<0$ we
are in the previous situation. Indeed, with the change of variable 
$x\rightarrow -x$ and the change of function 
$w(x)\rightarrow -w(x)$, we have that the new ratio of functions $\tilde{h}(x)=-y(-x)/w(-x)$ is such that 
$\tilde{h}(\alpha^+ )>0$ and the 
previous results hold in the interval $[\alpha ,\beta]=[-b , -a]$.  Then, we can write a common result for both cases.
We only give the result corresponding to Theorem \ref{casopos}.

\begin{theorem}
\label{casogeneral}
Let $h(x)$ be a solution of (\ref{Riccaini}) with continuous coefficients and $d(x)>0$, $e(x)>0$. Suppose that
either $h(a^+)>0$ or that $h(b^-)<0$ and take $s=+$, $c=a^+$ in the first case and $s=-$, $c=b^{-}$ in 
the second. Then, $h(x)$ does not change sign in $(a,b)$, and 
if the characteristic root $\lambda^{s}(x)$ is monotonic and $\lambda^{s \prime}(c)h'(c)>0$ then
$$ 
\left(\left|h(x)\right|-\left|\lambda^s (x)\right|\right)\Frac{d\lambda^s}{dx}<0\,\, \forall x\in (a,b)
$$
\end{theorem}

\begin{remark}
The condition $\lambda^{s \prime} (c)h'(c)>0$ is equivalent to $$ 
\left(\left|h(c)\right|-\left|\lambda^s (c)\right|\right)\Frac{d\lambda^s}{dx}(c)<0$$
\end{remark}

\section{Bounds for first order DDEs}
\label{DDEpar}

Now, consider a first order difference-differential equation (\ref{DDESta})
and assume it holds for $n\ge n_0$ and that 
that the shift $n\rightarrow n+1$ is possible (true for continuous dependence 
on the parameter $n$). 
Then the solutions of (\ref{DDESta}) are also solutions of a three-term
recurrence relation
\begin{equation}
\label{TTRRR}
e_{n+1} y_{n+1}(x)+(b_{n+1} (x)-a_{n}(x)) y_{n}(x)-d_n y_{n-1}(x)=0 .
\end{equation}
As in the previous section, we assume $d_n(x)e_n(x)>0$. 

Let $\bar{\lambda}_n^{\pm}$ be the roots of the algebraic equation 
\begin{equation}
e_{n+1}\bar{\lambda}_n^2 +(b_{n+1}-a_n) \bar{\lambda_n } -d_n=0,
\end{equation}
that is:
\begin{equation}
\label{etarec}
\begin{array}{l}
\bar{\lambda}_n^{\pm } = R_n E_n (-\bar{\eta}_n\pm \sqrt{1+\bar{\eta}_n^2}),\\
R_n=\sqrt{d_n/e_n},\,E_n =\sqrt{e_n/e_{n+1}},\, \bar{\eta}_n=(b_{n+1}-a_n)/(2\sqrt{d_n e_{n+1}})
\end{array}
\end{equation}

If $\lim_{n\rightarrow +\infty}\bar{\eta}_n \neq 0$ 
then 
$\lim_{n\rightarrow +\infty}|\bar{\lambda}_n^{+}/\bar{\lambda}_n^{-}|\neq 1$, and if the coefficients are of algebraic growth as a function of 
$n$, Perron-Kreuser theorem (see \cite[Thm 4.5]{Gil:2007:NSF}) states that independent
 pairs of solutions $\{y^{(1)}_{k},y^{(2)}_{k}\}$ exist such that 
\begin{equation}
\label{pairs}
\lim_{n\rightarrow +\infty}\Frac{1}{\bar{\lambda}_{n}^{+}}\Frac{y_n^{(1)}}{y_{n-1}^{(1)}}=1,
\,\lim_{n\rightarrow +\infty}\Frac{1}{\bar{\lambda}_n^{-}}\Frac{y_n^{(2)}}{y_{n-1}^{(2)}}=1 .
\end{equation}

If $\bar{\eta}_n>0$ the minimal solution is $y_n^{(1)}$ and $y_{n}^{(2)}$
is dominant, and therefore $\lim_{n\rightarrow +\infty}y_n^{(1)}/y_n^{(2)}=0$ .  If $\eta_n<0$ the roles are reversed.
In both cases we have, for sufficiently large $n$,
$y_{n+1}^{(1)}y_n^{(1)}>0$ and $y_{n+1}^{(2)}y_n^{(2)}<0$.

\begin{remark}
The minimal solution satisfies $\bar{\eta}_n y_n/y_{n-1}>0$ for large $n$, while the dominant solutions are such that
$\bar{\eta}_n y_n/y_{n-1}<0$ for large $n$.
\end{remark}

Notice that the roots (\ref{etarec}) are closely related to the characteristic roots of the Riccati
equation (\ref{root}):
\begin{equation}
\label{rootricn}
\lambda_n^{\pm}(x)=\displaystyle\sqrt{\Frac{d_n (x)}{e_n(x)}}(-\eta_n (x)\pm\sqrt{1+\eta_n(x)^2}),\,\eta_n (x)=
\Frac{b_n(x)-a_n(x)}{2\sqrt{d_n (x) e_n (x)}}.
\end{equation}
As we have shown in the previous section, when $\lambda_n^{\pm}(x)$ are monotonic they provide bounds for some solutions.
On the other hand, if 
$\lim_{n\rightarrow +\infty}\bar{\lambda}_n^{\pm}/\lambda_n^{\pm}=1$ the function ratios have these bounds
as limits. This explains why the bounds (\ref{rootricn}) tend to be sharper as $n$ becomes larger.
Because of this, we refer to these bounds as Perron-Kreuser bounds.

In section \ref{PKboun} we will obtain additional upper and lower sharp bounds starting from the bounds of Theorem \ref{casopos} and using the three-term recurrence.

Before this, it is important to stress that for the Perron-Kreuser bounds to hold, it is crucial that the characteristic roots are monotonic as
a function of $x$.
This, however, is a quite general situation, as we next see.

\subsection{Monotonicity of the characteristic roots}

The next result relates the monotonicity properties of the characteristic roots with monotonicity 
properties as a function of $n$ which are known to hold for a large set of functions.

\begin{theorem}
\label{monolam}
Let $y_{k}(x)$, $k=n,n-1$, be solutions of second order ODEs $y^{\prime\prime}_k (x)+ B_k(x) y^{\prime}_k (x) +A_k (x) y_k (x)=0$, with $A_k(x)$, $B_k(x)$
continuous in $(a,b)$ and $B_n(x)=B_{n-1}(x)$. Assume that $y_n (x)$ and $y_{n-1}(x)$ satisfy a system (\ref{DDESta})
with $d_n (x) e_n(x)>0$ and differentiable coefficients. Then, $e_n(x)/d_n(x)$ is constant as a function of $x$, and 
if $A_n (x)\neq A_{n-1}(x)$ the characteristic
roots $\lambda_n^{\pm}(x)$ (\ref{rootricn})
are monotonic in $(a,b)$. Furthermore, $d\lambda_n^{\pm}(x)/dx$ has the same sign as $A_{n-1}(x)-A_n(x)$ and 
$-\eta_n^{\prime}(x)$.
\end{theorem}

\begin{proof}
Differentiating the first equation of the system (\ref{DDESta}) and eliminating $y_{n-1}$ and proceeding similarly with the second equation we have
\begin{equation}
\begin{array}{l}
y_k'' (x) +B_k (x) y_k ' (x) +A_k (x) y_k (x)=0, \,k=n,n-1,
\end{array}
\end{equation}
with coefficients satisfying:
\begin{equation}
\label{loscoefs}
\begin{array}{l}
B_n(x)- B_{n-1}(x)=\Frac{e_n '(x)}{e_n(x)}-\Frac{d_n '(x)}{d_n(x)},\\
A_n (x) - A_{n-1} (x)= b_n '(x) -a_n '(x) -b_n(x) \Frac{e_n'(x)}{e_n(x)} + a_n(x) \Frac{d_n'(x)}{d_n(x)}
\end{array}
\end{equation}
Now, because we are assuming that $B_n (x) =B_{n-1}(x)$ the first equation implies that $d_n(x)/e_n(x)$ does not depend on $x$.
Therefore, from the expression of the characteristic roots (\ref{etarec}) we see that $d\lambda_n^{\pm}(x)/dx$ has the same sign as $-\eta_n '(x)$.
All that remains to be proved is that $A_n(x)-A_{n-1}(x)$ has the same sign as $\eta_n '(x)$. But considering the second equation
of (\ref{loscoefs}) and using that $d_n'(x)/d_n(x)=e_n '(x)/e_n(x)$ one readily sees that $A_n(x) -A_{n-1}(x)=2\sqrt{d_n(x) e_n(x)}\eta_n '(x)$, which
proves the theorem.

\end{proof}

\begin{remark}
\label{2otro}
If $e_n(x)d_n(x)<0$ and $\eta_n(x)^2>1$, it is also true that both roots are monotonic if $B_n(x)=B_{n-1}(x)$ and 
$A_n(x)\neq 
A_{n-1}(x)$, but $\lambda_n^{+}(x)\lambda_n^{-}(x)>0$ and $\lambda_n^{+ \prime}(x)\lambda_n^{- \prime}(x)<0$ in this case. 
\end{remark}

The case described in Theorem \ref{monolam} is, for instance, the situation for Bessel functions, parabolic cylinder functions and
the classical orthogonal polynomials when $n$
is the degree of the polynomials. 

\subsection{Perron-Kreuser bounds}

\label{PKboun}

In the following, we assume that $\eta_{n}(x)$, 
$\bar{\eta}_n (x)$, 
$d_n(x)$, $e_n(x)$ and $h_n (x)=y_{n}(x)/y_{n-1}(x)$  do not 
change sign for large enough $n$ (say $n\ge n_0$). Notice that 
the sign condition for $h_n(x)$ is satisfied for large enough $n$ when Perron-Kreuser theorem 
holds. An immediate application of Theorem \ref{casogeneral} gives:

\begin{theorem}[First Perron-Kreuser bound]
\label{1stP}
Let $d_n (x)>0$, $e_n (x)>0$ and $h_n (x)=y_n (x)/y_{n-1}(x)$ with constant sign for $n\ge  n_0$ and for
any $x\in (a,b)$. Let $s=\mbox{sign}(h_n (x))$ and 
$\lambda_n^{s}(x)$ as in Eq. (\ref{rootricn}). Then, if $h_n(a^+)>0$ and $h_n '(a^+)\lambda_n^{s \prime}(a^+)>0$
or  $h_n (b^-)>0$ and $h_n '(b^-)\lambda_n^{s \prime}(b^-)>0$ the following holds in $(a,b)$:
\begin{equation}
\left(|h_n(x)|-F_n^s (x)\right)\lambda_n^{s \prime}(x)<0,\,n \ge n_0
\end{equation}
\begin{equation}
\label{primobp}
F_n^s(x)
=R_n(x)(-s\eta_n(x)+\sqrt{1+\eta_n(x)^2})
=\Frac{R_n(x)}{s\eta_n(x)+\sqrt{1+\eta_n(x)^2}}
\end{equation}
\end{theorem}

Further bounds can be obtained by iteration of (\ref{TTRRR}), which we write:
\begin{equation}
\label{remi}
\Frac{y_n(x)}{y_{n-1}(x)}=d_n \left(b_{n+1}-a_n+e_{n+1}\Frac{y_{n+1}(x)}{y_n (x)}\right)^{-1}
\end{equation}
it is clear that for minimal solutions ($\bar{\eta}_n (x)y_n (x)/y_{n-1} (x)>0$ for large $n$), by
substituting $y_{n+1}(x)/y_n (x)$ by a lower (upper) bound we get an upper (lower) bound for $y_{n}(x)/y_{n-1} (x)$. 
We only give the first iteration.

\begin{theorem}[Second Perron-Kreuser bound for minimal solutions]
\label{PK21}
Under the conditions of Theorem \ref{1stP} and if $s\bar{\eta}_{n}>0$, $s=\mbox{sign}(h_n)$  then
\begin{equation}
\left(|h_n (x)|-S_n^{s +}\right) \lambda_n^{s \prime}(x) > 0 ,\,n \ge n_0
\end{equation}
where 
\begin{equation}
\label{secobp}
S_n^{s +}=\Frac{D_n E_n R_{n}}{s (2 D_n \bar{\eta}_n -\eta_{n+1})
+\sqrt{1+\eta_{n+1}^2}}
\end{equation}
$D_n =\sqrt{d_n/d_{n+1}}$; $E_n$, $R_n$ and $\bar{\eta}_n$ given by (\ref{etarec}) and $\eta_n$ by
(\ref{rootricn}).
\end{theorem}

The second superscript of the notation $S_n^{s +}$ stands for the sign of $\bar{\eta}_{n}y_n/y_{n-1}$.

Notice that theorem \ref{PK21} may be true for $n=n_0-1$ too, because Theorem \ref{1stP} is used in the proof with the shift $n\rightarrow n+1$.

The similarity of the second expression of (\ref{primobp}) with (\ref{secobp}) indicates that for coefficients of algebraic
growth we will generally have $\lim_{n\rightarrow +\infty}F_n^s/S_n^{s+}=1$.

Further iterations are possible and this gives a convergent sequence of upper and lower bounds under the conditions of Theorem \ref{1stP} and
\ref{PK21} and provided that Perron-Kreuser theorem holds (which  implies that the recurrence admits a minimal solutions). We don't prove this result,
but the convergence of the sequence of bounds for the minimal solution 
follows immediately by using the same arguments considered in \cite{Segura:2011:BRM} for the case of Modified Bessel functions of the first kind.

We can also obtain additional bounds for dominant solutions by writing
\begin{equation}
\label{reme}
\Frac{y_{n}(z)}{y_{n-1} (x)}=-\Frac{b_n-a_{n-1}}{e_n}+\Frac{d_{n-1}}{e_n}\Frac{y_{n-2}(z)}{y_{n-1} (z)}
\end{equation}
Differently from the case of minimal solutions, the sequence of bounds is not a convergent sequence.
We give an explicit formula for the first iteration:

\begin{theorem}[Second Perron-Kreuser bound for dominant solutions]
\label{PK22}
Under the conditions of Theorem \ref{1stP} and if $s \bar{\eta}_{n-1}<0$, $s=\mbox{sign}(h_n)$,
\begin{equation}
(|h_n(x)|-S_n^{s-})\lambda_n^{s \prime}(x) >0,\,n \ge n_0 +1
\end{equation}
where
\begin{equation}
\label{secobp2}
S_n^{s -}=D_{n-1}E_{n-1}R_n\left(-s(2 E_{n-1} ^{-1}\bar{\eta}_{n-1} -\eta_{n-1})+\sqrt{1+\eta_{n-1}^2}\right)
\end{equation}
\end{theorem}

Notice that the previous theorem can only be guaranteed to be true for $n=n_0+1$, because Theorem \ref{1stP} is used 
in the proof with the shift $n\rightarrow n-1$.

The similarity of the first expression in (\ref{primobp}) with (\ref{secobp2}) is clear. For coefficients of algebraic
growth we will generally have $\lim_{n\rightarrow +\infty}F_n^s/S_n^{s-}=1$.

\subsection{Tur\'an-type inequalities}
\label{turanin}

Because upper and lower bounds are available for $|y_{n}/y_{n-1}|$ both when $y_n$ is a minimal
or a dominant solution (Theorems \ref{1stP}, \ref{PK21} and \ref{PK22}), upper and lower bounds
for $|y_{n}/y_{n-1}||y_n /y_{n+1}|$ become available. The modulus can be skipped if $y_n/y_{n-1}$
does not change sign (as assumed earlier). With this:

\begin{equation}
l_n\le L_n (x)<\Frac{y_n(x)}{y_{n+1}(x)}\Frac{y_n (x)}{y_{n-1} (x)}
< U_n (x)\le u_n,
\end{equation}
where $l_n=\min_{x}\{L_n (x)\}$ and $u_n=\max_x\{U_n (x)\}$.
Many new Tur\'an-type inequalities are found in section \ref{examples} by using this simple idea.

\subsection{Bounds of Liouville-Green type}

Using the difference-differential system (\ref{DDESta}) and the Perron-Kreuser bounds, bounds on the logarithmic 
derivatives can be established. We give the bounds obtained from the first Perron-Kreuser bound.

\begin{theorem}
\label{LGchain}
Under the hypothesis of Theorem  \ref{1stP} and if $d\lambda_n^{s}/dx>0$ ($s={\mbox sign} (y_n (x)/y_{n-1}(x))$):
\begin{equation}
\label{LGD}
 s \Frac{y_{n-1}^{\prime}(x)}{y_{n-1}(x)} <s\Frac{a_n(x)+b_n (x)}{2}+\sqrt{d_n(x)e_n(x)}\sqrt{1+\eta_n (x)^2}<
 s \Frac{y_{n}^{\prime}(x)}{y_{n}(x)}
\end{equation}
If $d\lambda_n^{s}/dx<0$ the inequalities are reversed
\end{theorem}

 Two consequences follow. First, we observe that the ratios $y^{\prime}_{k}(x)/y_{k}(x)$
are monotonic as a function of the discrete variable $k$. Second, because we are
assuming that the shift $n \rightarrow n+1$ is possible, 
we have both an upper and a lower bound for $y^{\prime}_n/y_{n}$. Upper and lower bounds could also be obtained by
considering both the first and second Perron-Kreuser bounds. 

 In the examples we will see that these bounds, after integrating the logarithmic derivative, are related to the Liouville-Green 
approximation for solutions of
second order ODEs. In fact, using this analysis and by Liouville-transforming the first order system associated
to the ODE $y''(x)+A(x)y(x)=0$, conditions can be established under which the
LG approximation for the solutions the ODE $y''(x)+A(x)y(x)=0$ are bounds for some of the solutions. We leave
this analysis for a future paper.

\section{Applications}
\label{examples}

We give a number of examples of application of the techniques described in the paper. 
We concentrate
mainly on the case $d_n(x)e_n(x)>0$.
We also give two examples of application for monotonic systems with $d_n(x) e_n (x)<0$. 
The examples given by no means exhaust the functions for which the analysis is possible.

\subsection{Cases with $d_n(x)e_n(x)>0$} 
\label{posiex}

 We give examples which include classical orthogonal
polynomials outside their interval of orthogonality. 
In all cases except the last one, Theorem \ref{monolam} holds.
The last case is that of Laguerre functions of negative argument, for which
 Theorem \ref{monolam} can not be applied but the characteristic roots are still monotonic and
the same analysis is therefore possible. 
Some monotonicity properties for the determinants of some of these functions 
(modified Bessel functions, Hermite polynomials of imaginary order and Laguerre polynomials of negative argument) 
were considered in \cite{Ismail:2007:MPD}.

\subsubsection{Modified Bessel functions}

These are solutions of $x^2 y''+x y'-(x^2+\nu^2)y=0$.  This was the case considered 
in detail in \cite{Segura:2011:BRM}, and most of the results obtained in that paper are direct consequences of
the more general results of the present one.

\subsubsection{Parabolic cylinder functions}
\label{PCFs}

The parabolic cylinder
function $U(n,x)$ is a solution of the differential equation $y''(x)-(x^2/4+n)y(x)=0$, with coefficient 
$A(x)=-(x^2/4+n)$ depending monotonically on the parameter $n$ (Theorem \ref{monolam} holds).

Considering the DDE satisfied by $U(n,x)$ \cite[12.8.2-3]{Olv:2010:NIST} and defining 
$y_n (x)=e^{i\pi n} U(n,x)$ \footnote{It is not important that the new functions are complex, because
we are dealing with ratios; an alternative definition could be $y_n (x)=(-1)^{\lfloor n\rfloor } U (n,x)$.} we have:
\begin{equation}
\label{DDEUT}
\begin{array}{l}
y^{\prime}_{n}(x) =\Frac{x}{2}y_n (x)+ y_{n-1}(x),\\
y^{\prime}_{n-1}(x)=-\Frac{x}{2}y_{n-1}(x)+(n-1/2)y_{n}(x) .
\end{array}
\end{equation}
where $n$ will be real and positive.
For this system
\begin{equation}
\eta_n (x)=-\Frac{x}{2 \sqrt{n-1/2}},\,\bar{\eta}_n (x)=\eta_{n+1}(x),\lambda_n^{\pm}(x)=\Frac{-2}{x\mp\sqrt{4n-2+x^2}}
\end{equation}

From \cite[12.9.1]{Olv:2010:NIST} we have $h_n(+\infty)=0^{-}$ and $h_n'(+\infty)=0^+$ and
because $\lambda_n^{-}(+\infty)=0^+$ then theorem \ref{casogeneral} holds, as well as
theorems \ref{1stP} and \ref{PK21}. Therefore

\begin{theorem} 
\label{uaxuno}
For $n>1/2$ and $x\ge 0$ the following holds
\begin{equation}
\label{primeupos}
\Frac{2}{x+\sqrt{4n+2+x^2}}<\Frac{U(n,x)}{U (n-1,x)}<\Frac{2}{x+\sqrt{4n-2+x^2}}
\end{equation}

The lower bound also holds if $n\in (-1/2,1/2)$ and it turns to an equality if $n=-1/2$.
\end{theorem}

The lower bound is obtained from the upper bound and the application
of the three-term recurrence relation: if $B_m(n,x)$ is a positive upper (lower) bound for $U(n,x)/U(n-1,x)$, $x>0$,
then
\begin{equation}
\label{ultimaiter}
B_{m+1}(n,x)=1/(x+(n+1/2) B_m(n+1,x))
\end{equation}
is a lower (upper) bound for the same ratio.
The process can be continued as $m\rightarrow +\infty$
and the sequence 
is convergent (because $U(n,x)$ is minimal).

Now, consider $y_n(x)=U(n,-x)$, which is also solution of (\ref{DDEUT}). 
Using the values of $U(n,0)$ and $U'(n,0)$ 
\cite[12.2.6-7]{Olv:2010:NIST} it is easy to prove that
$h_n (0^+)>0$, $h_n^{\prime}(0^+)>0$, $n>1/2$, $x\ge 0$ and then 
$h_n^{\prime}(0^+)d\lambda^{+}_n (0^+)/dx>0$ and 
Theorem \ref{casopos} holds. The corresponding Perron-Kreuser bounds (theorems 
\ref{1stP} and \ref{PK22}), give:

\begin{theorem} 
\label{cuaxsos}
For $n>3/2$ and $x\ge 0$ the following holds
\begin{equation}
\label{primeuneg}
\Frac{x+\sqrt{4n-6+x^2}}{2n-1}<\Frac{U(n,-x)}{U (n-1,-x)}<\Frac{x+\sqrt{4n-2+x^2}}{2n-1}
\end{equation}
The upper bound is also valid if $n\in (1/2,3/2)$. 

\end{theorem}
The upper bound in (\ref{primeupos}) has the same expression as (\ref{primeuneg}) but with $x$ in 
replaced by $-x$. Therefore:
\begin{remark}
Theorems \ref{uaxuno} and \ref{cuaxsos} hold for all real $x$, but for $x<0$ the lower bound of Theorem \ref{uaxuno}
only holds for all $x<0$ if $n>1/2$. The lower bounds are sharper when $x>0$.
\end{remark}

The following Tur\'an-type inequalities follows from Theorems \ref{uaxuno} and \ref{cuaxsos}

\begin{theorem}
Let $F(x)=U(n,x)^2/(U(n-1,x) U (n+1,x))$. 

The following holds for all real $x$:
\begin{equation}
\sqrt{\Frac{n-3/2}{n+1/2}}<\Frac{n-1/2}{n+1/2}F(x)<1<F(x)<\sqrt{\Frac{n+3/2}{n-1/2}}
\end{equation}
The first inequality holds for $n>3/2$ and the rest for $n>1/2$. For $x<0$ the third inequality
 also holds if $n\in (-1/2,1/2)$.
\end{theorem}

Finally, considering Theorem \ref{LGchain} and writing together the results for $U(n,x)$ and $U(n,-x)$ 
we have the next result.
\begin{theorem}
For all real $x$ and $n\ge 1/2$ the following holds:
\begin{equation}
\label{conbouU}
-\sqrt{x^2/4+n+1/2}<\Frac{U^{\prime} (n,x)}{U(n,x)}<-\sqrt{x^2 /4+n-1/2}
\end{equation}
The left inequality also holds for $n>-1/2$.
\end{theorem}

These type of bounds are useful for studying the attainable accuracy of methods
for computing the functions. In \cite{Gil:2006:CRP}, the following estimation for large
$x$ and/or $n$ was considered
for the condition number with respect to $x$:
\begin{equation}
C_x (U(a,x))=\left|x U^{\prime}(a,x)/U(a,x)\right|\sim x\sqrt{x^2/4+a},
\end{equation}
and similarly for $V(a,x)$. The bounds (\ref{conbouU}) prove that this a good estimation because it lies
between the upper and lower bounds.
From the previous discussion on the $V(a,x)$ function, one can prove that similar
bounds are valid for moderate $x$ ($x>1$ is enough); we consider later this function.

Integrating (\ref{conbouU}) we have
\begin{equation}
\begin{array}{l}
F_{n+1/2}(x)/F_{n+1/2}(y)<\Frac{U(n,y)}{U(n,x)}<F_{n-1/2}(x)/F_{n-1/2}(y),\\
F_{\alpha}(x)=\exp\left(\frac{x}{2}\sqrt{x^2/4+\alpha}\right)\left(x+2\sqrt{x^2/4+\alpha}\right)^{\alpha}
\end{array}
\end{equation}
and, in particular,
\begin{equation}
\label{bountot}
F_{n+1/2}(x)<\Frac{U(a,x)}{U(a,0)}<F_{n-1/2}(x)
\end{equation}
where
\begin{equation}
F_{\alpha}(x)=\exp\left(-\Frac{x}{2}\sqrt{\Frac{x^2}{4}+\alpha}\right)
\left(\Frac{x}{2\sqrt{\alpha}}+\sqrt{\Frac{x^2}{4\alpha}+1}\right)^{-\alpha}
\end{equation}

The bounds (\ref{bountot}) are useful for obtaining the range of parameters
for which function values are computable within the arithmetic capabilities of a
computer (overflow and underflow limits). These results confirms the estimations based on the Liouville-Green approximation
used in \cite{Gil:2006:ARP}.

\paragraph{$\circ$ Iterated coerror functions and Mill's ratio:}

In particular, considering Theorem \ref{uaxuno} and the relation of parabolic cylinder functions $U (n+1/2,x)$ with
the iterated coerror functions $i^n \mbox{erfc}(x)$ \cite[12.7.7]{Olv:2010:NIST}, $n\in {\mathbb N}$, the
following follows:
\begin{equation}
M_{n+1} (x)<\Frac{i^n \mbox{erfc}(x)}{i^{n-1} \mbox{erfc}(x)}<M_n (x),\,n=1,2,...\,;M_n(x)=(x+\sqrt{2n+x^2})^{-1} .
\end{equation}
These inequalities appear in \cite{Amos:1973:BIC}.

Theorem \ref{uaxuno} also gives bounds on Mill's ratio ($n=1/2$). From lower bound
in Theorem (\ref{uaxuno}) and the upper bound obtained by iterating with (\ref{ultimaiter}) we have

\begin{theorem}
Let $r(x)=e^{x^2/2}\int_{x}^{+\infty}e^{-t^2 /2}dt$, then
\begin{equation}
\label{boundis}
\Frac{2}{x+\sqrt{x^2+4}}< r(x) < \Frac{4}{3x+\sqrt{x^2+8}}
\end{equation}
\end{theorem}

The lower bound was obtained in \cite{Birn:1942:IMR} and the upper bound in  \cite{Sam:1953:SIM}.
In our case, these results follow from a more general result. See also \cite{Bar:2008:MRM} for 
an alternative proof.

Further iterations (see (\ref{ultimaiter})) give additional sharper bounds: 
\begin{theorem}
\begin{equation}
R_{2k+1}<r(x)
<R_{2k}(x)
\end{equation}
\begin{equation}
R_n(x)=\Frac{1}{x+}\Frac{1}{x+}\Frac{2}{x+}
\ldots\Frac{n}{T_{n}(x)},\,T_n=(x+\sqrt{4n +x^2})/2
\end{equation}
\end{theorem}

where, as usual we denote $\frac{1}{a+}\frac{1}{b+}\ldots=1/(a+1/(b+\ldots))$

\paragraph{$\circ$ Hermite polynomials of imaginary variable}

A similar analysis to that for $U(n,-x)$ can
be carried for the PCF $V(n,x)$. Indeed, $y_n(x)=V(n,x)/\Gamma (n+1/2)$ is a solution of (\ref{DDEUT})
and 
$h_n(x)=y_n (x)/y_{n-1}(x)$ is such that $h_n(0^+)>0$. Two situations take place depending on
the values of $n$. First, if $n\in (2k-1,2k)$, $k\in{\mathbb N}$, then $h_n^{\prime}(0^+)>0$ and the upper bound of 
Theorem \ref{cuaxsos} holds in this case and for all $x>0$ while the lower bound will hold for $n\in (2k,2k+1)$.
Contrarily, if $n\in (2k,2k+1)$ then $h_n^{\prime}(0^+)<0$, while $h_n (+\infty)>0$, and the upper bound
only holds for large enough $x$; a similar situations occurs with the lower bound when $n\in (2k-1,2k)$.

We only consider the first case. Then, using the relation of $V(n+1/2,x)$, $n\in {\mathbb N }$, 
with Hermite polynomials \cite[12.7.3]{Olv:2010:NIST} we get:

\begin{theorem}
\begin{equation}
\Frac{V(n,x)}{V (n-1,x)}<\Frac{x+\sqrt{4n-2+x^2}}{2},\, x>0,\, n\in (2k-1,2k),\,k\in {\mathbb N}
\end{equation}
\begin{equation}
\Frac{x+\sqrt{4n-6+x^2}}{2}<\Frac{V(n,x)}{V (n-1,x)},\, x>0,\, n\in (2k,2k+1),\,k\in {\mathbb N}
\end{equation}
\begin{equation}
-i\Frac{H_{2k+1}(ix)}{H_{2k}(ix)}<x+\sqrt{4k+2+x^2},\, x>0,\,k =0,1,2\ldots
\end{equation}
\begin{equation}
i\Frac{H_{2k-1}(ix)}{H_{2k}(ix)}<(x+\sqrt{4k-2+x^2})^{-1},\, x>0,\,k\in{\mathbb N}
\end{equation}
\begin{equation}
\Frac{H_{2k}(ix)^2}{H_{2k-1}(ix)H_{2k+1}(ix)}>\sqrt{\Frac{k-1/2}{k+1/2}},\,k\in {\mathbb N}, x\in {\mathbb R}.
\end{equation}
\end{theorem}

Hermite polynomials
of imaginary argument were also considered in \cite{Ismail:2007:MPD}.
The well-known Tur\'an-type inequality for Hermite polynomials \cite{Sko:1954:OIT} 
$H_n(x)^2-H_{n-1}(x) H_{n+1}(x)>0$, $x\in {\mathbb R}$, does not hold on the imaginary axis, but a similar
property $H_{n}(ix)^2-\sqrt{(n-1)/(n+1)}H_{n-1}(ix)H_{n+1}(ix)>0$ holds true for all $x>0$ if $n$ is even.

\subsubsection{Oblate Legendre functions}
\label{OLFs}

These are Legendre functions of imaginary argument, which are functions appearing in the solution of
Dirichlet problems in oblate spheroidal coordinates \cite{Gil:1998:CEP}. Denoting
\begin{equation}
p_{n}(x)=e^{-i n\pi/2}P_{n}^{m}(ix)
\end{equation}
and using the differential relations 
\cite[14.10.4-5]{Olv:2010:NIST} we have
\begin{equation}
\begin{array}{l}
p_{n}^{\prime}(x)=\Frac{1}{1+x^2}\left\{n x p_{\nu} (x)+ (n+m) p_{\nu-1}(x)\right\}\\
p_{n-1}^{\prime}(x)=\Frac{1}{1+x^2}\left\{-n x p_{\nu -1} (x)+ (n-m) p_{\nu }(x)\right\}
\end{array}
\end{equation}
and $q_{n} (x)=Q_{n}^{m}(ix)$, $Q_n^m$ being the second kind Legendre function, satisfies the same system. We consider $n>m$ and $x>0$. 
This is again an example for which Theorem \ref{monolam}
holds. The roles played in this case by the functions $Q_n^m (ix)$ and $P_n^m(ix)$ are very similar 
to the roles of $U(n,x)$ and $V(n,x)$ in the previous section. We omit details and only summarize
the main results.

\begin{theorem} The following holds for $x>0$ and real $n>m>0$
\begin{equation}
0<i\Frac{Q_n^m (ix)}{Q_{n-1}^m (ix)}<\Frac{n+m}{n}\left[x+\sqrt{1+x^2-\Frac{m^2}{n^2}}\right]^{-1}<\sqrt{\Frac{n+m}{n-m}}
\end{equation}
\begin{equation}
i\Frac{Q_n^m (ix)}{Q_{n-1}^m (ix)}>\Frac{n+m}{nx+(n+1)\sqrt{1+x^2-\Frac{m^2}{(n+1)^2}}}
\end{equation}
\begin{equation}
1<\Frac{n+m+1}{n+m}\Frac{Q_n^m (ix)}{Q_{n-1}^m (ix) Q_{n+1}^m (ix)}<\sqrt{\Frac{(n+2)^2-m^2}{n^2-m^2}}
\end{equation}
\end{theorem}

\begin{theorem} The following holds for $x>0$ and integer $n,m$, $n>m$:
\begin{equation}
\label{primeP}
0<-i\Frac{P_n^m (ix)}{P_{n-1}^m (ix)}<\Frac{n}{n-m}\left[x+\sqrt{1+x^2-\Frac{m^2}{n^2}}\right],\, n-m \mbox{ odd}
\end{equation}
\begin{equation}
\label{secondP}
\Frac{1}{n-m}\left[nx+(n-1)\sqrt{1+x^2-\Frac{m^2}{(n-1)^2}}\right]<-i\Frac{P_n^m (ix)}{P_{n-1}^m (ix)},\, n-m \mbox{ even}
\end{equation}
\begin{equation}
\Frac{P_n^m (ix)^2}{P_{n-1}^m (ix) P_{n+1}^m (ix)}<1+\Frac{1}{n-m},\, n-m \mbox{ odd}
\end{equation}
\end{theorem}

For $m=0$ we have Legendre polynomials. If $n$ is odd, we have 
$P_n (ix)^2<0$ and therefore 
$P_n(ix)^2-(1+1/n)P_{n-1}(ix)P_{n+1}(ix)>0$. It appears, as numerical experiments show, 
that in this case the same Tur\'an inequality that holds in the real interval $(-1,1)$ \cite{Sze:1948:ITC} also holds
in the imaginary axis if $n$ is odd: $P_n(ix)^2-P_{n-1}(ix)P_{n+1}(ix)>0$; the same
is not true if $m\neq 0$.

\subsubsection{Laguerre functions of negative argument}

Next we consider an example for which Theorem \ref{monolam} can not be applied but the analysis is 
possible because the characteristic roots are monotonic. 

Consider the Laguerre functions $y_{\nu ,\alpha}(x)=L_{\nu}^{\alpha}(-x)$, $x>0$.
Using well known recurrences
and differentiation formulas, we have
\begin{equation}
\label{ricaL}
\begin{array}{l}
y_{\nu+1 ,\alpha -1}^{\prime}(x)=y_{\nu,\alpha}(x)\\
x y_{\nu ,\alpha}^{\prime}(x)= -(\alpha +x) y_{\nu ,\alpha}(x)+ (\nu+1)y_{\nu +1,\alpha -1}
\end{array}
\end{equation} 
and 
\begin{equation}
\label{recuL}
(\nu +1) y_{\nu+1,\alpha-1}(x)=(\alpha+x)y_{\nu,\alpha }(x)+x y_{\nu -1,\alpha +1}
\end{equation}
Considering \cite[Theorem 2]{Seg:2009:NSS} it follows that $y_{\nu ,\alpha}$ is a dominant
solution of the recurrence (\ref{recuL}) in the direction of increasing $\nu$ (and decreasing $\alpha$).

With $h(x)=y_{\nu,\alpha}(x)/y_{\nu+1 ,\alpha -1}(x)$, the positive characteristic root $\lambda^+(x)$ of the associated Riccati equation 
turns out to be increasing if $\nu>-1$ and $\alpha>0$. On the other hand, it is easy to check that for these values
$h(0^{+})>0$ and $h^{\prime}(0^{+})>0$.
Theorem \ref{casopos} holds and $\lambda^+ (x)$ is a bound:

\begin{theorem}
\label{upLa}
For any $\alpha>0$, $\nu>-1$ and $x>0$ the following holds
\begin{equation}
0<\Frac{L_{\nu+1}^{\alpha-1}(-x)}{L_{\nu}^{\alpha}(-x)}<\Frac{\alpha +x +\sqrt{(\alpha+x)^2+4(\nu+1)x}}{2(\nu+1)}
\end{equation}
\end{theorem}

On the other hand, from the recurrence (\ref{recuL}) we have 
\begin{equation}
\Frac{L_{\nu}^{\alpha}(-x)}{L_{\nu +1}^{\alpha -1}(-x)}=\left(\Frac{\alpha+x}{\nu+1}+\Frac{x}{\nu +1}
\Frac{L_{\nu-1}^{\alpha +1}(-x)}{L_{\nu}^{\alpha }(-x)}\right)^{-1}
\end{equation}
and from this we obtain the second Perron-Kreuser bound:

\begin{theorem}
\label{loLa}
For any $\alpha>-1$, $\nu>0$ and $x>0$ the following holds
\begin{equation}
\Frac{L_{\nu+1}^{\alpha-1}(-x)}{L_{\nu}^{\alpha}(-x)}>\Frac{\alpha+x-1+\sqrt{(\alpha+x+1)^2+4\nu x}}{2(\nu +1)}
\end{equation}
\end{theorem}

And from these bounds we get the following Tur\'an-type inequalities: 

\begin{theorem}
For any $\nu\ge 0$ and $\alpha\ge 0$, $x>0$ the following holds:
\begin{equation}
\Frac{\nu}{\nu+1}\Frac{\alpha}{\alpha+1}
<\Frac{L_{\nu+1}^{\alpha-1}(-x)}{L_{\nu}^{\alpha}(-x)}\Frac{L_{\nu-1}^{\alpha+1}(-x)}{L_{\nu}^{\alpha}(-x)}<\Frac{\nu}{\nu+1}
\end{equation}
\end{theorem}

A second 
independent 
solution of (\ref{ricaL}) 
which is a minimal solution of (\ref{recuL}) as
$\nu\rightarrow +\infty$ 
follows from \cite[Theorem 2]{Seg:2009:NSS}. 
Bounds can be also obtained for this solution. We omit the details. 

Other bounds and inequalities can be obtained using other recursions or using relations between contiguous functions.
For example, 
using \cite[18.9.13]{Olv:2010:NIST}, we have:
\begin{equation}
\Frac{L_{\nu+1}^{\alpha -1}(x)}{L_{\nu}^{\alpha}(x)}=1+
\Frac{L_{\nu+1}^{\alpha}(x)}{L_{\nu}^{\alpha}(x)}
\end{equation} 
and upper and lower bounds for $L_{n}^{\alpha}(-x)/L_{n-1}^{\alpha} (-x)$ follow from the previous results. As a consequence of this
new bounds, one can prove the following
\begin{theorem}
\begin{equation}
\Frac{\nu}{\nu+1}<
\Frac{L_{\nu -1}^{\alpha}(-x)}{L_{\nu}^{\alpha}(-x)}\Frac{L_{\nu +1}^{\alpha}(-x)}{L_{\nu}^{\alpha}(-x)}
<\Frac{\nu}{\nu+1}\Frac{\nu+\alpha+1}{\nu+\alpha-1}
\end{equation}
where the first inequality holds for $\nu>0$, $\alpha>-1$ and the second for $\nu>0$, $\nu+\alpha>1$.
\end{theorem}

For positive $x$, it is known that $L_{n-1}^{\alpha}(x)L_{n+1}^{\alpha}(x)/L_n^{\alpha}(x)^2<1$
\cite{Sko:1954:OIT}. For
negative argument we have an upper bound greater that $1$, which suggests that the Tur\'an-type inequality
for positive $x$
 does not hold for negative $x$, as numerical experiments show.

\subsection{Two examples with $d_n (x) e_n(x)<0$}
\label{monocondit}

The DDEs corresponding to a pair $\{ p_n (x),p_{n-1}(x)\}$ of 
classical orthogonal polynomials satisfy $d_n (x)e_n (x)<0$ in their interval
of orthogonality because this is a necessary condition for oscillation
\cite[Lemma 2.4]{Segura:2002:ZSF}. However,
for values of the variable for which the polynomials are free of zeros, one can expect that 
$\eta_n(x)^2>1$ 
and that the DDE becomes monotonic ($\eta_n(x)^2<1$  is also a necessary condition for oscillation \cite[Thm. 2.1]{Gil:2003:CZT}). 
This is the case of Laguerre and Hermite
polynomials for large enough $x>0$. We consider these two examples.

\subsubsection{Hermite polynomials}

Hermite polynomials satisfy
\begin{equation}
\begin{array}{l}
H_n^{\prime}(x)=2n H_{n-1}(x),\\
H_{n-1}^{\prime}(x)=2 xH_{n-1}-H_n(x)
\end{array}
\end{equation}
We have $\eta_n (x)=x/\sqrt{2n}$ and $\eta_n (x)>1$ if $x>\sqrt{2n}$ (monotonic case). The characteristic roots are both
of them positive
\begin{equation}
\lambda_n^{\pm}(x)=x\pm\sqrt{x^2-2n}.
\end{equation}

Defining $h_{n}(x)=H_{n}(x)/H_{n-1}(x)$
we have that $h_n (+\infty)=+\infty$ and $h_n^{\prime} (+\infty)>0$ because the coefficient of degree $n$ of $H_n(x)$ is positive. 
Then $h_n(x)>\lambda_n^{+}(x)$
for enough $x>0$ because $h_n'(x)>0$ only if $h_n(x)<\lambda_n^{-}(x)$ or $h_n(x)>\lambda_n^{+}(x)$, but 
$h_n(+\infty)>\lambda_n^{-}(+\infty)=0^+$. Then, we have that $h_n(x)>\lambda_n^+ (x)$ for large $x$. And
 because $\lambda_n^{+ \prime}(x)>0$ if $x>\sqrt{2n}$, 
then, necessarily:

\begin{equation}
\label{primebounH}
h_n(x)=\Frac{H_{n}(x)}{H_{n-1}(x)}>x+\sqrt{x^2-2n},x\ge \sqrt{2n} .
\end{equation}
We can iterate the recurrence relation. Contrary to the case $e_n (x) d_n (x)>0$, we will not obtain
sequences of lower and upper bounds, but only lower bounds. Writing
\begin{equation}
h_{n+1}(x)=2x-2n/h_n(x)
\end{equation}
and using (\ref{primebounH}) we get a lower bound for $h_{n+1}(x)$. We shift the parameter $n$ and get
\begin{equation}
\label{seconbounH}
h_n(x)>x+\sqrt{x^2-2(n-1)},x\ge\sqrt{2(n-1)} .
\end{equation}
\noindent
This improves Eq. (\ref{primebounH}) and enlarges the range of validity of the bound with respect to $x$, but
reduces the range of validity with respect to $n$ ($n\ge 2$).

The next iteration gives a bound for $n\ge 3$:
\begin{equation}
\begin{array}{ll}
h_n(x)>&F(n,x),x\ge\sqrt{2(n-2)}\\
& F(n,x)=(n-2)^{-1}[(n-3)x+(n-1)\sqrt{x^2-2(n-2)}]
\end{array}
\end{equation}

Taking into account the largest zero of $H_n(x)$ is larger than the largest zero of $H_{n-1}(x)$ the previous bounds
give bounds on the largest zero of $H_n(x)$. We see that the largest zero of $H_n(x)$ is smaller than $\sqrt{2(n-k)}$ if
$n>k$.

We consider just one more iteration and get
\begin{equation}
h_n(x)\ge 2x-2(n-1)/F(n-1,x)=G(n,x),\,x>\sqrt{2(n-3)}
\end{equation}
and if $G(n,\sqrt{2(n-3)})> 0$ then $G(n,x)> 0$ if $x>\sqrt{2(n-3)}$, and 
the largest zero will be smaller than $\sqrt{2(n-3)}$; this condition is
met if $n\ge 7$. 
A sharper bound has recently appeared in the literature \cite{Dim:2010:SBE} valid for all $n$. However,
the result is sharper than previous results, like for instance those in \cite{Area:2004:ZGH}, which is interesting given the
simplicity of the analysis. This reflects the fact that the bounds on function ratios 
(our main topic) are
sharp. 

\subsubsection{Laguerre polynomials}

We give some results for Laguerre polynomials omitting details. 
Defining $h_n^\alpha (x)=-L_{n}^{\alpha}(x)/L_{n-1}^{\alpha}(x)$, we have
$h_n^\alpha (+\infty)=+\infty$ and $h^{\alpha\prime}_n(+\infty)=+\infty$ and, proceeding similarly as before:
\begin{equation}
\begin{array}{ll}
2n h_n^\alpha (x)
>& x-(2n+\alpha)+\sqrt{(x-2n-\alpha)^2-4n(n+\alpha)},\\
& x\ge 2n+\alpha+2\sqrt{n(n+\alpha)}
\end{array}
\end{equation}
and after the first iteration of the recurrence we have:
\begin{equation}
\begin{array}{l}
2n h_n^{\alpha} (x)
> f(x),\,  x\ge 2n^{*}+\alpha+2\sqrt{n^*(n^*+\alpha)},\,n^*=n-1,\\
 f(x)=x-(2n +\alpha)+\sqrt{(x-2n^{*}-\alpha)^2-4n^{*}(n^{*}+\alpha)} .
\end{array}
\end{equation}

This proves that the largest zero of $L_{n}^{\alpha}(x)$ is smaller than $x^{*}=2n+\alpha-2+\sqrt{(n-1)(n-1+\alpha)}$,
provided that $f(x^*)>0$, which is true if $\alpha>(n-1)^{-1}-(n-1)$, $n\ge 2$; notice that values $\alpha< -1$ are
allowed for large enough $n$. The bound in \cite{Ism:1992:BEZ} is slightly sharper, and is improved in \cite{Dim:2010:SBE}. 

Further iterations are possible, but not so easy to analyze. The next iteration will give a bound
\begin{equation}
\label{otrobl}
2n h_n^\alpha (x)
> g(x),\,  x\ge 2(n-2)+\alpha+2\sqrt{(n-2)(n-2+\alpha)}=x^{*}
\end{equation}
$x^{*}$ is an upper bound for the largest zero provided that $g(x^*)>0$. This condition is met for a larger
$\alpha$ range as $n$ becomes larger. For $n\ge 10$, this holds for any $\alpha >-1$. The bound 
(\ref{otrobl}) is of more limited in terms of $n$
but numerical experiments show that it is sharper than the
bound in \cite{Dim:2010:SBE} for $\alpha \le 12$

We expect that lower bounds for the smallest zero can be also obtained with a similar analysis.

The main message, as before, is that the bounds on function ratios are sharp for large $x$ because they give the correct asymptotic
behavior as $x\rightarrow +‌\infty$, but also for moderate $x$ given the sharpness on the bounds
on the largest zero.



\section*{Acknowledgements}
This work was supported by  {\emph{Ministerio de Ciencia e Innovaci\'on}}, 
project MTM2009-11686.

\bibliographystyle{unsrt}
\bibliography{biblio}

\end{document}